\documentclass[12pt]{amsart}
\usepackage{latexsym, url}
\usepackage{amsthm}
\usepackage{amsmath}
\usepackage{amsfonts}
\usepackage{amssymb}
\usepackage[dvips]{graphicx}
\usepackage{xypic}
\input xy
\xyoption{all}

\newtheorem{theo}{Theorem}[section]

\newtheorem{propo}[theo]{Proposition}

\newtheorem{rem}[theo]{Remark}

\newcommand\Mod{\operatorname{Mod}}

\newcommand\id{\operatorname{id}}

\newcommand\Set{\operatorname{\bf Set}}
\newcommand\Met{\operatorname{\bf Met}}

\newcommand\Ban{\operatorname{\bf Ban}}
\newcommand\Norm{\operatorname{\bf Norm}}
\newcommand\CAlg{\operatorname{\bf CAlg}}
\newcommand\BanAlg{\operatorname{\bf BanAlg}}
\newcommand\IBanAlg{\operatorname{\bf IBanAlg}}
\newcommand\Alg{\operatorname{\bf Alg}}

\newcommand\CCAlg{\operatorname{\bf CCAlg}}

\newcommand\CMet{\operatorname{\bf CMet}}
\newcommand\MonCMet{\operatorname{\bf MonCMet}}

\newcommand\reals{\mathbb {R}}
\newcommand\ca{\mathcal {A}}

\newcommand\ck{\mathcal {K}}

\newcommand\pa{\parallel}

 \newbox\noforkbox \newdimen\forklinewidth
\forklinewidth=0.3pt \setbox0\hbox{$\textstyle\smile$}
\setbox1\hbox to \wd0{\hfil\vrule width \forklinewidth depth-2pt
 height 10pt \hfil}
\wd1=0 cm \setbox\noforkbox\hbox{\lower 2pt\box1\lower
2pt\box0\relax}

\date{May 6, 2021}
 
\begin{document}
\title[Are Banach spaces monadic?]
{Are Banach spaces monadic?}
\author[J. Rosick\'{y}]
{J. Rosick\'{y}}
\thanks{Supported by the Grant Agency of the Czech Republic under the grant 
               19-00902S} 
\address{
\newline J. Rosick\'{y}\newline
Department of Mathematics and Statistics,\newline
Masaryk University, Faculty of Sciences,\newline
Kotl\'{a}\v{r}sk\'{a} 2, 611 37 Brno,\newline
Czech Republic}
\email{rosicky@math.muni.cz}

\begin{abstract}
We will show that Banach spaces are monadic over complete metric spaces via the unit ball functor. For the forgetful functor, one should take complete pointed metric spaces.
\end{abstract}
\maketitle

\section{Introduction}
We will present some results about monadicity of the category $\Ban$ of (complex) Banach spaces and linear maps of norm $\leq 1$. Of course, the question depends on a forgetful functor we choose. We cannot consider the underlying set functor $V_0:\Ban\to\Set$ because it does not preserve products. One should consider the unit ball functor $U_0:\Ban\to\Set$ whose left adjoint is $l_1:\Set\to\Ban$. There is well known that $U_0$ is not monadic and its monadic completion is the category of totally convex spaces (see \cite{PR}). However, using \cite{PR}, we show that the unit ball functor $U:\Ban\to\CMet$ is monadic where $\CMet$ is the category of complete metric spaces and nonexpanding maps. The forgetful functor $V:\Ban\to\CMet$ does not preserve products again.

The category $\CMet$ has many deficiencies and it is natural to replace it by the category $\CMet_\infty$ of generalized complete metric spaces by allowing distances to be $\infty$ while keeping all other requirements, as well as the type of morphisms (see, e.g., \cite{RT}). In the same way, we can generalize Banach spaces by allowing norms to be $\infty$. We will show that the forgetful functor $V_\infty:\Ban_\infty\to\CMet_\infty$ is now monadic. The reason is that $\CMet_\infty$ is symmetric monoidal closed and generalized Banach spaces coincide with monoids in $\CMet_\infty$ equipped with scalar multiplication. 

Another modification of $\CMet$ is the category $\CMet^\bullet$ of pointed complete metric spaces. Here, the forgetful functor $V^\bullet:\Ban\to\CMet^\bullet$ has a left adjoint (given by Lipschitz-free spaces, see, e.g., \cite{CDW}). We will show that $V^\bullet$ is monadic, which was suspected by T. Fritz in \cite{F}. Finally, we will touch the question of monadicity of the category $\CAlg$ of $C^\ast$-algebras over $\Ban$. 

We recall that a category $\ck$ is locally $\lambda$-presentable, where $\lambda$ is a regular cardinal, if it is cocomplete and has a set $\ca$ of $\lambda$-presentable objects
such that very object of $\ck$ is a $\lambda$-directed colimit of objects fro $\ca$.
Here, $\lambda$-directed colimits are colimits over $\lambda$-directed posets
and an object $A$ is $\lambda$-presentable if its hom-functor $\ck(A,-):\ck\to\Set$ preserves $\lambda$-directed colimits. A category is locally presentable if it is locally $\lambda$-presentable for some regular cardinal $\lambda$. All needed facts about locally presentable categories can be found in \cite{AR}. 

In what follows, forgetful functors on various categories of Banach spaces will be denoted
by $V$ with needed decorations and, similarly, unit ball functors will be denoted by $U$.

\noindent {\bf Acknowledgement.} We are grateful to the referee for valuable comments and suggestions.

\section{Generalized Banach spaces}
The category $\Met$ of metric spaces and nonexpanding maps is neither complete nor cocomplete, and the tensor product $X\otimes Y$, which puts the $+$-metric 
$$d\otimes d((x,y),(x',y'))=d(x,x')+d(y,y')
$$ 
on $X\times Y$, fails to make $\Met$ monoidal closed. (Note that $X\otimes Y$ must not be confused with the Cartesian product $X\times Y$ in $\Met$, which is given by the max-metric.) One therefore enlarges $\Met$ to the category $\Met_{\infty}$ of {\em generalized metric spaces}, by allowing distances to be $\infty$ while keeping all other requirements, as well as the type of morphisms. Then $\Met_{\infty}$ is complete and cocomplete and monoidal closed, with the internal hom providing the hom-set  $\Met_{\infty}(X,Y)$ with the sup-metric $d(f,g)=\sup\{d(fx,gx)\;|\;x\in X\}$. Moreover, $\Met_\infty$ is locally $\aleph_1$-presentable (see \cite{LR} 4.5(3)). The category $\CMet_\infty$ of complete generalized metric spaces is locally $\aleph_1$-presentable too (see \cite{AR1} 2.3(2)).  
 
The category $\Ban$ of (complex) Banach spaces and linear maps of norm $\leq 1$ is locally $\aleph_1$-presentable (see \cite{AR} 1.48). We will also consider the category $\Ban_\infty$ of \textit{generalized Banach spaces}, by allowing norms to be $\infty$ while keeping all other requirements, as well as the type of morphisms. Similarly, $\Norm_\infty$ will be the category generalized normed spaces and linear maps of norm $\leq 1$.  

\begin{propo}\label{lp}
The category $\Ban_\infty$ is locally $\aleph_1$-presentable.
\end{propo}
\begin{proof}
Consider the single-sorted signature with unary relation symbols $R_r$ for each $0\leq r\in\reals$, constant $0$, binary operation $+$ and unary operations $c\cdot-$ for $c\in\Bbb C$. Let $T$ consist of complex vector space axioms and axioms
$$
(\forall x)(R_0(x)\leftrightarrow x=0)
$$
for all $r\leq s$
$$
(\forall x)(R_r(x)\rightarrow R_s(x))
$$
for all $r,s$
$$
(\forall x,y)(R_r(x)\wedge R_s(y)\rightarrow R_{r+s}(x+y))
$$
for all $r$
$$
(\forall x) R_r(x)\leftrightarrow R_{|c|r}(c\cdot x)
$$
for $r_0\geq r_1\geq\dots r_n\geq\dots$ with $r=\lim r_n$
$$
(\forall x) (\bigwedge_n R_{r_n}(x)\rightarrow R_r(x))
$$
Since $T$ is a universal Horn theory in $L_{\omega_1,\omega}$, the category $\Mod(T)$ of $T$-models and homomorphisms is locally $\aleph_1$-presentable (see \cite{AR} 5.30). If we interpret $R_r(a)$ as $\pa a\pa\leq r$, $\Mod(T)$ is isomorphic to the category $\Norm_\infty$. 
 
$\Ban_\infty$ is a reflective subcategory of $\Norm_\infty$ closed under $\aleph_1$-directed colimits (reflection is given by the completions). Hence $\Ban_\infty$ is locally $\aleph_1$-presentable (see \cite{AR} 1.39).
\end{proof}

\begin{theo}\label{monad}
The forgetful functor $V_\infty:\Ban_\infty\to\CMet_\infty$ is monadic.
\end{theo}
\begin{proof}
Since $V_\infty$ preserves limits and $\aleph_1$-directed colimits, it has a left adjoint $F_\infty$ (see \cite{AR} 1.66). Let $T_\infty=V_\infty F_\infty$ be the induced monad. Given a generalized Banach space $A$, the operation $+:V_\infty A\otimes V_\infty A\to V_\infty A$ is nonexpanding because 
\begin{align*}
d((x,y),(x',y')) &= d(x,x') + d(y,y') = \pa x-x'\pa + \pa y-y'\pa\geq \pa x-x' + y-y'\pa\\
&=d(x+y,x'+y').
\end{align*}
Hence $V_\infty A$ is a monoid in $\CMet_\infty$, i.e., a complete metric space $M$ equipped with operations $+:M\otimes M\to M$ and $0:I\to M$, where $I$ is the one-point metric space, satisfying the monoid axioms. Hence $V=V_2V_1$ where $V_1:\Ban_\infty\to\MonCMet_\infty$ and $V_2:\MonCMet_\infty\to\CMet_\infty$ are forgetful functors. Following \cite{P}, the category $\MonCMet_\infty$ of monoids in $\CMet_\infty$ is monadic over $\CMet_\infty$.

In a monoid $M$ over $\CMet_\infty$, we define $\pa x\pa=d(x,0)$. We have
$$
\pa x+y\pa = d(x+y,0) \leq d((x,y),(0,0)) = d(x,0) + d(y,0) =\pa x\pa + \pa y\pa.
$$
Hence generalized Banach spaces coincide with monoids in $\CMet_\infty$ equipped with unary operations $c\cdot -$ satisfying the appropriate axioms. These operations are nonexpanding iff $|c|\leq 1$. Since $c\cdot-$ is inverse for $c^{-1}\cdot-$, it is easy to see that the category $\Ban_\infty$ is equivalent to the category of monoids in $\CMet_\infty$ equipped with nonepanding operation $c\cdot-$ for $|c|\leq 1$ satisfying the apropriate axioms.

Consider the functor $H:\CMet_\infty\to\CMet_\infty$ sending $X$ to 
$(X\otimes X)\amalg I\amalg\coprod_c X$ where $c\in \Bbb C$, $|c|\leq 1$. Then $H$-algebras are generalized complete metric spaces equipped with operations $+$, $0$ and $c\cdot-$ for $|c|\leq 1$. Since $H$ preserves directed colimits, \cite{AR} Remark 2.75 implies that the category $H$-$\Alg$ of $H$-algebras is locally presentable. Following \cite{AP} 5.6, the forgetful functor to $H$-$\Alg\to\CMet_\infty$ creates all colimits that $H$
preserves. In particular, it creates absolute coequalizers and, following Beck's theorem,
it is monadic (see \cite{ML}). Like in \cite{P}, the category of $H$-algebras satifying the appropriate axioms is closed in $H$-$\Alg$ under directed and $V_\infty$-absolute colimits. Thus $\Ban_\infty$ is monadic.
\end{proof}

\begin{rem}\label{directed}
{
\em
(1) The functor $V_\infty$ even preserves directed colimits. Indeed, if $x=\lim_n x_n$ and $y=\lim_n y_n$ then $x+y=\lim_n (x_n+y_n)$. Hence the monad $T_\infty$ preserves directed colimits. 

(2) The value $F_\infty(1)$ of the left adjoint $F_\infty:\CMet_\infty\to\Ban_\infty$ is the generalized Banach space $\Bbb C_\infty$ of complex numbers where all non-zero elements have norm $\infty$.
}
\end{rem}

\section{Complete pointed metric spaces}
A \textit{pointed generalized metric space} $(X,0)$ is a generalized metric space $X$ with a choosen element $0\in X$. Morphisms of pointed generalized metric spaces are nonexpanding maps preserving $0$. Let $\Met^\bullet_\infty$ be the category of pointed generalized metric spaces and $\CMet^\bullet_\infty$ the category of pointed generalized complete metric spaces. The categories $\Met^\bullet_\infty$ and $\CMet^\bullet_\infty$
are locally $\aleph_1$-presentable. They are also symmetric monoidal closed where the tensor product is the \textit{smash product} $X\wedge Y$. Recall that $X\wedge Y$ is the pushout
$$
		\xymatrix@=2pc{
			X\otimes 1\amalg 1\otimes Y \ar [r]^{}\ar[d]_{} & X\otimes Y \ar[d]^{}\\
			1  \ar [r]_{}& X\wedge Y
		}
		$$ 
where $1=\{0\}$ is the zero object in $\Met^\bullet_\infty$. The internal hom provides the hom-set $\Met^\bullet_\infty(X,Y)$ with the sup-metric.	

The category $\Met^\bullet$ of pointed metric spaces is a coreflective full subcategory of $\Met_\infty^\bullet$ where the coreflector assigns to a pointed generalized metric space $A$ its subspace consisting of all elements $a$ such that $d(0,a)<\infty$. Like in the proof of \ref{lp}, $\Met^\bullet$ is locally $\aleph_1$-presentable. Similarly, the category $\CMet^\bullet$ of pointed complete metric spaces is locally $\aleph_1$-presentable. 	

\begin{theo}\label{monad1}
The forgetful functor $V^\bullet:\Ban\to\CMet^\bullet$ is monadic.
\end{theo}
\begin{proof}
Like in the proof of \ref{monad}, the functor $V^\bullet$ preserves limits and $\aleph_1$-directed colimits, thus it has a left adjoint $F^\bullet$. Let $T^\bullet=V^\bullet F^\bullet$ be the induced monad and $W:\CMet^\bullet\to\CMet$ the forgetful functor. Then $WV^\bullet=V$ is the forgetful functor $\Ban\to\CMet$ which is the domain-codomain restriction of $V_\infty$.

Consider a pair $f,g:A\to B$ of morphisms in $\Ban$ such that $V^\bullet f,V^\bullet g$ has a split coequalizer in $\CMet^\bullet$ given by $h:V^\bullet B\to C$, $s:C\to V^\bullet B$ and $t:V^\bullet B\to V^\bullet A$. Then $V_\infty f,V_\infty g:V_\infty A\to V_\infty B$ has a split coequalizer in $\CMet$ given by $Wh:V_\infty B\to WC$, $Ws$ and $Wt$. Following \ref{monad} and Beck's theorem, there is a unique $\bar{C}$ and a unique $\bar{h}:B\to\bar{C}$ such that $V_\infty\bar{C}=WC$ and $V_\infty\bar{h}=Wh$ and, moreover, $\bar{h}$ is a coequalizer of $f$ and $g$. We have $WV^\bullet\bar{C}=V_\infty\bar{C}=WC$. This means that $V^\bullet W\bar{C}$ and $WC$ are the same metric spaces. Since $0$ in $V_\infty\bar{C}$ is $0\in \bar{C}$ and $0$ in $C$ is $0\in \bar{C}$, $V^\bullet\bar{C}=C$. Since $W$ is faithful and $WV^\bullet\bar{h}=V_\infty\bar{h}=Wh$, we have $V^\bullet\bar{h}=h$. Thus $V^\bullet$ creates the coequalizer of $V^\bullet f$ and $V^\bullet g$. Following Beck's theorem, $V^\bullet$ is monadic. 
\end{proof}

\begin{rem}
{
\em
(1) Banach spaces are not monoids in $\CMet^\bullet$ because $+$ is not a morphism $V^\bullet A\wedge V^\bullet A\to V^\bullet A$.

(2) The left adjoint $F^\bullet$ sends a pointed complete metric space $X$ to its Lipschitz-free space (see, e.g., \cite{CDW}), they are also called Arens-Eells spaces.

(3) Like in \ref{directed}, $T^\bullet$ preserves directed colimits.

(4) In the same way as in \ref{monad1} we show that the forgetful functor $V_\infty^\bullet:\Ban_\infty\to\CMet_\infty^\bullet$ is monadic.
}
\end{rem}

\section{Banach spaces}
The forgetful functor $V:\Ban\to\CMet$ does not preserve products and we have to take the unit ball functor $U:\Ban\to\CMet$. Recall that the unit ball functor $U_0:\Ban\to\Set$ is not monadic and its monadic completion is the category of totally convex spaces (see \cite{PR}). These are algebras with operations indexed by sequences $(c_i)^\infty_{i=0}$ of complex numbers satisfying $\sum^\infty_{i=0} |c_i|\leq 1$. These operations are denoted by $\sum^\infty_{i=0} c_ix_i$ and are called totally convex operations. They satisfy some equations (see \cite{PR}, or \cite{AR} 1.48). Over $\CMet$, it suffices to take only finitary totally convex operations indexed by sequences $c_1,\dots,c_n$ such that $\sum^n_{i=0} |c_i|\leq 1$. Their algebras are called \textit{finitely totally convex spaces}.

A totally convex space $A$ is \textit{separated} if for $a_1,a_2\in A$ and $c\in\Bbb C$, $0<|c|<1$,
$ca_1=ca_2$ implies $a_1=a_2$ (see \cite{PR1} 11.2). 
\begin{theo}\label{main}
The unit ball functor $U:\Ban\to\CMet$ is monadic.
\end{theo}
\begin{proof}
Like in the proof of \ref{monad}, the functor $U$ preserves limits and $\aleph_1$-directed colimits, thus it has a left adjoint $F$. Let $T=UF$ be the corresponding monad. In the same way as in \cite{PR} 1.3, we prove that the induced functor $\Ban\to\Alg(T)$ is fully faithful. We have to prove that it is essentially surjective.

Let $(A,h)$ be a $T$-algebra. Since $h\eta_A=\id_A$, $\eta_A$ isometrically embeds $A$ to the unit ball $UFA$. Unit balls of Banach spaces are totally convex and totally convex operations $(UB)^\infty\to UB$ have norm $\leq 1$. They are also preserved by morphisms in $\Ban$. We define totally convex operations on $A$ by 
$$
\sum^\infty_{i=0}c_i a_i = h\sum^\infty_{i=0}c_i\eta_A (a_i).
$$
Following the equation $hT(h)=h\mu_A$, where $\mu$ is the multiplication of $T$, $(A,h)$ is a totally convex space. 

We will show that $A$ is  separated. Assume that $a_1,a_2\in A$ and $ca_1=ca_2$ for all $0<|c|<1$. Hence
$$
h(c\eta_A(a_1))=ca_1=ca_2=h(c\eta_A(a_2))
$$ 
for all $0<|c|<1$. Since $1=\lim_n c_n$, $0<c_n<1$ and $h$ is continuous, 
$$
a_1=h\eta_A(a_1)=h(\lim_n c_n\eta_A(a_1))=h(\lim_n c_n\eta_A(a_2))=h\eta_A(a_2)=a_2.
$$
 
The unit ball functor makes $\Ban$ a full reflective subcategory of the category of totally convex spaces (see \cite{PR}, 7.7). Let $\sigma_A:A\to UA^\ast$ be the unit of this reflection. Since $A$ is separated, $\sigma_A:A\to UA^\ast$ is an isometry (see \cite{PR1} 11.3). Following \cite{PR}, 7.3, the image of $\sigma_A$ contains the open unit ball of $A^\ast$. 

Assume that there is $a\in A^\ast$ which does not belong to this image $\sigma_A(A)$. 
Put $s_n=\sum_{k=1}^n\frac{1}{2^k}$ and $a_n=s_na$. Then $a=\lim_n a_n$ and $a_n\in\sigma_A(A)$. Thus $a_n=\sigma_A(b_n)$ where $b_n\in A$. Since $\frac{a_n}{s_n}=a=\frac{a_m}{s_m}$, we have $a_m=\frac{s_m}{s_n}a_n$. Since 
$\frac{s_m}{s_n}\leq 1$ for $m\leq n$ and $\sigma_A$ is a morphism of totally convex
spaces, we have 
$$
h\eta_A(b_m)=b_m=\frac{s_m}{s_n}b_n=h(\frac{s_m}{s_n}\eta_A(b_n))
$$ 
for $m\leq n$. Since $h$ is nonexpanding, we have
\begin{align*}
d(b_m,b_n)&= d(h\eta_A(b_m),h\eta_A(b_n))=d(h(\frac{s_m}{s_n}\eta_A(b_n)),h\eta_A(b_n))
\leq d(\frac{s_m}{s_n}\eta_A(b_n),\eta_A(b_n))\\
&=\pa\frac{s_n-s_m}{s_n}\eta_A(b_n)\pa\leq\frac{s_n-s_m}{s_n}.
\end{align*}
Hence $b_1,b_2,\dots,b_n,\dots$ is a Cauchy sequence in $A$ and, since $A$ is complete,
it is converging to $b\in A$. Since $a=\sigma_A(b)\in\sigma_A(A)$, we get a contradiction.


Hence $(A,h)$ is a unit ball of a Banach space.
\end{proof}

\begin{rem}
{
\em
(1) Like in \ref{directed}, $T$ preserves directed colimits.

(2) The \textit{Kantorovich monad} $K$ on $\CMet$ is given by barycentric operations and its algebras are closed convex subsets of Banach spaces (see \cite{FP}). Our monad $T$ has more operations, i.e., less algebras. Following \cite{AMMU}, $K$ preserves directed colimits.

(3) The unit ball functor $\Ban\to\Met$ is monadic as well. Its $T$-algebras are complete being retracts of complete metric spaces - given by $h\eta_A=\id_A$. The rest is the same as in \ref{main}.
}
\end{rem}

\section{$C^\ast$-algebras}
Let  $\CAlg$ be the category of unital $C^\ast$-algebras and $\CCAlg$ the category of commutative unital $C^\ast$-algebras. The forgetful functor $G:\CAlg\to\Ban$ preserves limits, isometries and directed colimits (see \cite{DR} 6.10). Thus it has a left adjoint $F$. The same holds for the restriction $G_c:\CCAlg\to\Ban$ of $G$ on the category $\CCAlg$. The unit $\eta_B:B\to GFB$ is a linear isometry. Thus $F$ is faithful. In the commutative case, the left adjoint $F_c$ was described in \cite{S} and called the Banach-Mazur functor.

Let $\BanAlg$ be the category of unital Banach algebras and $\IBanAlg$ the category of unital involutive Banach algebras. The forgetful functors $G_0:\BanAlg\to\Ban$ and $G_1:\BanAlg\to\Ban$ again preserve limits, isometries and directed colimits. 
\begin{theo}
The forgetful functors $G_0:\BanAlg\to\Ban$ and $G_1:\IBanAlg\to\Ban$ are monadic.
\end{theo}
\begin{proof}
Let $T_0=G_0F_0$ and $T_1=G_1F_1$ be the induced monads. Given a unital Banach algebra $A$, the operation $\cdot:G_0A\otimes_pG_0A\to G_0A$ has norm $\leq 1$. Here, $\otimes_p$ is the projective tensor product on $\Ban$ which satisfies $\pa x\otimes_p y\pa = \pa x\pa\pa y\pa$ (see \cite{BS} 2.5.10). Hence
$$
\pa x\cdot y\pa\leq\pa x\pa\pa y\pa =\pa x\otimes_p y\pa.
$$
Hence $GA$ is a monoid in $\Ban$. In fact, $\BanAlg$ coincides with the category of monoids in $\Ban$. Since $X\otimes -$ has a right adjoint (see \cite{B} 6.1.9h),
the category of monoids in $\Ban $ is monadic over $\Ban$ (see \cite{P}).

If $A$ is an unital involutive Banach algebra $G_1A$ is also equipped with a unary operation $(-)^\ast$ which is of norm $\leq 1$ again. Hence unital involutive Banach algebras coincide with involutive monoids in $\Ban$. Like in \ref{monad}, $G_1$ is monadic.
\end{proof}

\begin{rem}
{
\em
(1) The same holds for commutative unital (involutive) Banach algebras.

(2) Unital $C^\ast$-algebras are unital involutive Banach algebras satisfying
$$
\pa x^\ast\cdot x\pa = \pa x\pa^2.
$$
They form a full reflective subcategory of $\IBanAlg$ where the reflector is given by the enveloping $C^\ast$-algebra. The unit ball functor $\CCAlg\to\Set$ is monadic (see \cite{PeR}). Similarly, the unit ball functor $\CAlg\to\Set$ is monadic (see \cite{VO}).

(3) Is $G:\CAlg\to \Ban$ monadic? Similarly, is $G_c:\CCAlg\to \Ban$ monadic?
}
\end{rem}

\end{document}